\documentclass[12pt]{amsart}
\usepackage[utf8]{inputenc}
\usepackage[utf8]{inputenc}
\usepackage[T1]{fontenc}
\usepackage[all]{xy}
\usepackage{lipsum}
\usepackage{url}
\usepackage{tikz}
\usepackage{stackrel}
\usepackage{color}
\usepackage{amsmath,amsthm,amssymb}

\usepackage{xcolor}

\usepackage{graphicx}

\theoremstyle{definition}

\newtheorem{theorem}{Theorem}[section]
\newtheorem{lemma}[theorem]{Lemma}
\newtheorem{example}[theorem]{Example}
\newtheorem{formulation}[theorem]{Formulation}
\newtheorem{proposition}[theorem]{Proposition}
\newtheorem{definition}[theorem]{Definition}

\newtheorem{remark}[theorem]{Remark}

\numberwithin{equation}{section}

\begin{document}

\renewcommand{\bf}{\bfseries}
\renewcommand{\sc}{\scshape}

\title[Transversal motion planning]%
{Transversal motion planning}

\author{Cesar A. Ipanaque Zapata}
\address{Departamento de Matem\'{a}tica, Universidade de S\~{a}o Paulo
Instituto de Matem\'{a}tica e  Estatística -IME/USP, R. do Matão, 1010 - Butantã, CEP:
05508-090 - S\~{a}o Paulo, Brasil}
\email{cesarzapata@usp.br}

\author{Fernando R. Chu Rivera} 
\address{Facultad de Ciencias Matemáticas - FCM-UNMSM,
Ciudad Universitaria - UNMSM, Av. República de Venezuela 3400, Cercado de Lima, Lima, Perú}
\email{fernando.chu@unmsm.edu.pe}


\thanks {The first author would like to thank grant\#2022/03270-8, S\~{a}o Paulo Research Foundation (FAPESP) for financial support.}

\begin{abstract} In this paper, we introduce the notion of transversal topological complexity (TTC) for a smooth manifold $X$ with respect to a submanifold of codimension 1 together with basic results about this numerical invariant. In addition, we present several examples of explicit transversal algorithms.  
\end{abstract}
\maketitle

\section{Introduction}\label{secintro}

Let $X$ be the space of all possible obstacle-free configurations or states of a given autonomous system. A \emph{motion planning algorithm on $X$} is a function which, to any pair of configurations $(C_1,C_2)\in X\times X$, assigns a continuous motion $\mu$ of the system, so that $\mu$ starts at the given initial state $C_1$ and ends at the final desired state $C_2$. The fundamental problem in robotics, \textit{the motion planning problem}, deals with how to provide, to any given autonomous system, with a motion planning algorithm. 

\medskip
For practical purposes, a motion planning algorithm should depend continuously on the pair of points $(C_1,C_2)$. Indeed, if the autonomous system performs within a noisy environment, absence of continuity could lead to instability issues in the behavior of the motion planning algorithm. Unfortunately, a (global) continuous motion planning algorithm on a space $X$ exists if and only if $X$ is contractible. Yet, if $X$ is not contractible, we could care about finding \emph{local} continuous motion planning algorithms, i.e., motion planning algorithms $s$ defined only on a certain open subset of $X\times X$, to which we refer as the domain of definition of $s$. In these terms, a \emph{motion planner on $X$} is a set of local continuous motion planning algorithms whose domains of definition cover $X\times X$. The \emph{topological complexity of $X$} \cite{farber2003topological}, TC$(X)$, is then the minimal cardinality among motion planners on $X$, while a motion planner on $X$ is said to be \emph{optimal} if its cardinality is TC$(X)$. The design of explicit motion planners that are reasonably close to optimal is one of the challenges of modern robotics (see, for example Latombe \cite{latombe} and LaValle \cite{lavalle}). 

\medskip In more detail, the components of the motion planning problem via topological complexity are as follows (see \cite{zapata2020}):
\begin{formulation} Ingredients in the motion planning problem via topological complexity:
\begin{enumerate}
    \item The obstacle-free configuration space $X$. The topology of this space is assumed to be fully understood in advance. 
    \item Query pairs $C=(C_1,C_2)\in X\times X$. The point $C_1{}\in X$ is designated as the initial configuration of the query. The point $C_2{}\in X$ is designated as the goal configuration.  
\end{enumerate}
In the above setting, the goal is to either describe a motion planner, i.e., describe
\begin{enumerate}\addtocounter{enumi}{2}
\item An open covering $U=\{U_1,\ldots,U_k\}$ of $X\times X$;
\item For each $i\in\{1,\ldots,k\}$, a local continuous motion planning algorithm, i.e., a continuous map $s_i\colon U_i\to X^{[0,1]}$ satisfying $$s_i(C)\left(j\right)=C_{j+1},\quad j=0,1$$ for any $C=(C_1,C_2){}\in U_i$ (here $X^{[0,1]}$ stands for the free-path space on $X$ joint with the compact-open topology),
\end{enumerate}
or, else, report that such an planner does not exist.
\end{formulation}

Let $X$ be a smooth manifold and $Z\subset X$ be a submanifold with codimension $1$.
\begin{definition}
    A path $\Gamma:[0,1]\to X$ is \textit{semi-transversal} to $Z$, denoted by $\Gamma\pitchfork_s Z$, if $\Gamma(t)\in Z$ with $t\in (0,1)$, implies that $\Gamma$ is smooth in $t$ and $\langle\Gamma'(t)\rangle+T_{\Gamma(t)}Z=T_{\Gamma(t)}X$ (equivalently, $\Gamma'(t)\notin T_{\Gamma(t)}Z$), where $\langle\Gamma'(t)\rangle$ is the subspace generated by $\Gamma'(t)$. Note that, if $\Gamma$ is transversal to $Z$ (we will denote $\Gamma\pitchfork Z$) then $\Gamma\pitchfork_s Z$. 

\end{definition}
\begin{example}
    The path $\Gamma\colon [0,1]\to\mathbb{R}^2$ given by $$\Gamma(t)=\begin{cases}
    (1-2t)(-1,0)+2t(0,1),& \hbox{for $0\leq t\leq 1/2$;}\\
    (2-2t)(0,1)+(1-2t)(1,0),&\hbox{for $1/2\leq t\leq 1$;}
    \end{cases}
    $$ is semi-transversal to $Z=\{(x,x):~x\in\mathbb{R}\}$ but is not transversal to $Z$. In fact, $\Gamma$ is not smooth in $t=1/2$.
\end{example}

$$
\begin{tikzpicture}[x=.9cm,y=.9cm]
\draw(-2,0)--(2,0);
\draw(0,-2)--(0,2);
\draw(-2,-2)--(2,2);
\draw[->](-1,0)--(0,1); 
\draw[->](0,1)--(1,0); 
\node [below] at (-1,0) {\tiny$(-1,0)$};
\node [below] at (1,0) {\tiny$(1,0)$};
\node [above] at (0,1) {\tiny$(0,1)$};
\node [above] at (1,1) {\tiny$Z$};
\end{tikzpicture}
$$

\medskip
Investigation of the problem of transversal motion planners for a robot, with state space $X$, leads us to study the pair $(X,Z)$, where $Z$ is a submanifold of $X$ with codimension $1$. The submanifold $Z$ may characterize some desired geometry for the motion planning algorithm.  A (local) transversal motion planning algorithm in $X$ with respect to $Z$ assigns to any pair of configurations $(C_1,C_2)$ in (an open set of) $X\times X$ a path of configurations \[\Gamma(t)\in X,~~t\in [0,1],\]  such that $\Gamma\left(i\right)=C_{i+1}$ for $i=0, 1$ and $\Gamma\pitchfork_s Z$.

\medskip
In this work we introduce the notion of transversal topological complexity together with basic results about this numerical invariant. Proposition~\ref{thm} together with Example~\ref{exam:non-transversal} give the motivation to introduce the tranversal topological complexity (Definition~\ref{def:ttc}). Examples~\ref{exam:transversal-rn} and \ref{exam:rd-spheres} present transversal algorithms in euclidean spaces. We define the notion of transversal LS category (Definition~\ref{def:tcat}) and present a lower bound for transversal complexity in terms of transversal LS category (Proposition~\ref{lower-bound}). Proposition~\ref{prop:difeo} shows an explicit construction of transversal algorithms through diffeomorphisms. Examples~\ref{exam:2-4} and \ref{exam:concrete-transversal-alg} show that submanifold $Z$ may imply some desired geometry for the motion planning algorithm.  

\section{Transversal topological complexity}
In this section we present the notion of transversal topological complexity (Definition~\ref{def:ttc}) together with basic results about this numerical invariant (Lemma~\ref{carac} and Propositions~\ref{lower-bound} and~\ref{prop:difeo}). Several examples are presented to illustrate the result arising in this field (Examples~\ref{exam:transversal-rn},~\ref{exam:rd-spheres},~\ref{exam:2-4} and~\ref{exam:concrete-transversal-alg}). Examples~\ref{exam:2-4} and \ref{exam:concrete-transversal-alg} show that submanifold $Z$ may imply some desired geometry for the motion planning algorithm.  

Let $X^{[0,1]}$ stand for the free-path space of a topological space $X$. Recall that Farber's topological complexity $\text{TC}(X)$ is the sectional category of the end-points evaluation fibration $e_{2}^X\colon X^{[0,1]}\to X\times X$, $e_{2}^X(\gamma)=(\gamma(0),\gamma(1))$. We use sectional category of a fibration $p\colon E\to B$ in the non reduced sense, i.e., it is the minimal number of open sets covering $B$ and on each of which $p$ admits a local continuous section.

 From the Whitney approximation theorem \cite[Theorem 6.26]{lee2012} we obtain the following statement.

\begin{lemma}\label{approximation}
   Let $X$ and $Y$ be smooth manifolds (without boundary), $f,g:X\to Y$ be smooth maps such that there is a continuous homotopy $H:X\times [0,1]\to Y$ with $H_0=f$ and $H_1=g$, then there exists a smooth homotopy $G:X\times [0,1]\to Y$ with $G_0=f$ and $G_1=g$. 
\end{lemma}

From \cite[Pg. 73]{pollack2010} we have the following statement.

\begin{lemma}\label{transversal}
Let $f:X\to Y$ be a smooth map and $Z\subset Y$ be a submanifold such that the boundary map $\partial f:\partial X\to Y$ is transversal to $Z$, then there exists a smooth map $g:X\to Y$ homotopic to $f$ such that $\partial g=\partial f$ and $g\pitchfork Z$.
\end{lemma}

Then, we have the following statement

\begin{proposition}\label{thm}
Let $X$ be a smooth manifold and $Z\subset X$ be a submanifold of codimension $1$, then $\text{TC}(X)$ coincides with the smallest positive integer $k$ for which  the product $X\times X$ is covered by $k$ open subsets $X\times X=U_1\cup\cdots\cup U_k$ such that for any $i=1,2,\ldots,k$ there exists a continuous section $s_i:U_i\to X^{[0,1]}$ of $e_2^X$ 
over $U_i$ (i.e., $e_2^X\circ s_i=incl_{U_i}$) and $\widetilde{s_i}\pitchfork Z$, where $\widetilde{s_i}\colon U_i\times [0,1]\to X$ is given by $\widetilde{s_i}(u)=s_i(u)(t)$ for any $u\in U_i$ and $t\in [0,1]$.     
\end{proposition}
\begin{proof}
    For $U\subset X\times X$ and $s:U\to X^{[0,1]}$ a continuous local algorithm, consider $\widetilde{s}:U\times [0,1]\to X,~\widetilde{s}((x,y),t)=s(x,y)(t)$. Note that, $\widetilde{s}$ is a continuos homotopy with $\widetilde{s}_0=\pi_1$ and $\widetilde{s}_1=\pi_2$, where $\pi_j$ is the projection to the $j$-th factor. By Lemma~\ref{approximation}, there exists a smooth homotopy $G:U\times [0,1]\to X$ with $G_0=\pi_1$ and $G_1=\pi_2$. Note that the boundary map $\partial G=\pi_1\sqcup \pi_2$ is transversal to $Z$, then by Lemma~\ref{transversal}, there exists a smooth map $F:U\times [0,1]\to X$ homotopic to $G$ such that $\partial F=\partial G$ (and thus, $F_0=\pi_1$ and $F_1=\pi_2$) and $F\pitchfork Z$. Then, the map $\sigma\colon U\to  X^{[0,1]}$ given by $\sigma(x,y)(t)=F((x,y),t)$ satisfies the conditions of the proposition. Thus, the proposition holds. 
\end{proof}

Note that, if $F:U\times [0,1]\to X$ is transversal to $Z$ with $U\subset X\times X$ does not implies that $F_u\colon [0,1]\to X,~t\mapsto F_u(t)=F(u,t)$, is transversal to $Z$ for any $u\in U$. To see this, we have the following example. 

\begin{example}\label{exam:non-transversal}
    Consider the pair $\left(\mathbb{R}^2,S^1\right)$ and the smooth map $F\colon \mathbb{R}^2\times \mathbb{R}^2\times[0,1]\to\mathbb{R}^2$ given by $F\left((C_1,C_2),t\right)=(1-t)C_1+tC_2$. Note that, $F\pitchfork S^1$ but $F_{\left((1,1),(-1,1)\right)}\colon [0,1]\to \mathbb{R}^2,~F_{\left((1,1),(-1,1)\right)}(t)=\left(1-2t,1\right)$, is not transversal to $S^1$.
\end{example}

Proposition~\ref{thm} says that Farber's topological complexity of $X$ coincides with the complexity of designing smooth homotopies $F:U\times [0,1]\to X$ with $F_0=\pi_1$ and $F_1=\pi_2$ such that $F\pitchfork Z$. However, Example~\ref{exam:non-transversal} motives the following definition. 

\begin{definition}\label{def:ttc}
The \textit{transversal topological complexity} TTC$(X,Z)$ of a path-connected smooth manifold (without boundary) $X$ with respect to a submanifold (without boundary) $Z$ with codimension 1 is the smallest positive integer TTC$(X,Z)=k$ for which  the product $X\times X$ is covered by $k$ open subsets $X\times X=U_1\cup\cdots\cup U_k$ such that for any $i=1,2,\ldots,k$ there exists a continuous section $s_i:U_i\to X^{[0,1]}$ of $e_2^X$ over $U_i$ (i.e., $e_2^X\circ s_i=incl_{U_i}$) and $s_i(x,y)\pitchfork_s Z$ for any $(x,y)\in U_i$. We call such a collection of local sections a \textit{transversal motion planner} with $k$ 
domains of continuity. If no such $k$ exists, we set TTC$(X,Z)=\infty$. 
\end{definition}

\begin{remark}
Note that $\text{TC}(X)\leq \text{TTC}(X,Z)$ for any smooth manifold $X$ and any submanifold $Z\subset X$ with codimension 1. 
\end{remark}

\begin{example}\label{exam:transversal-rn}
Consider the pair $(\mathbb{R}^{n}, \mathbb{R}^{n-1})$, with the latter understood as the submanifold $\mathbb{R}^{n-1} \times \{0\} \subset \mathbb{R}^{n}$. We claim that the transversal topological complexity of this pair is 1. Indeed, consider the (global) algorithm $s:\mathbb{R}^{n}\times \mathbb{R}^{n} \to (\mathbb{R}^{n})^{[0,1]}$ defined by $$s(C_1,C_2)(t)=\begin{cases}
        (1-2t)C_1+2te_n,&\hbox{for $0\leq t\leq 1/2$;}\\
       (2-2t)e_n+(2t-1)C_2,&\hbox{for $1/2\leq t\leq 1$;}\\
    \end{cases}$$ where $e_n=(0,\ldots,0,1)\in\mathbb{R}^{n}$. 
This map (see Figure~\ref{algorithm1}) is continuous and satisfies the transversal condition: If $s(C_1,C_2)(t)\in \mathbb{R}^{n-1}$ with $t\in (0,1)$, 
then $$s'(C_1,C_2)(t)=\begin{cases}
        2(e_n-C_1),&\hbox{for $0< t<1/2$;}\\
       2(C_2-e_n),&\hbox{for $1/2< t<1$;}\\
    \end{cases}$$ 
Hence, $s'(C_1,C_2)(t)\notin \mathbb{R}^{n-1}$. 
\end{example}

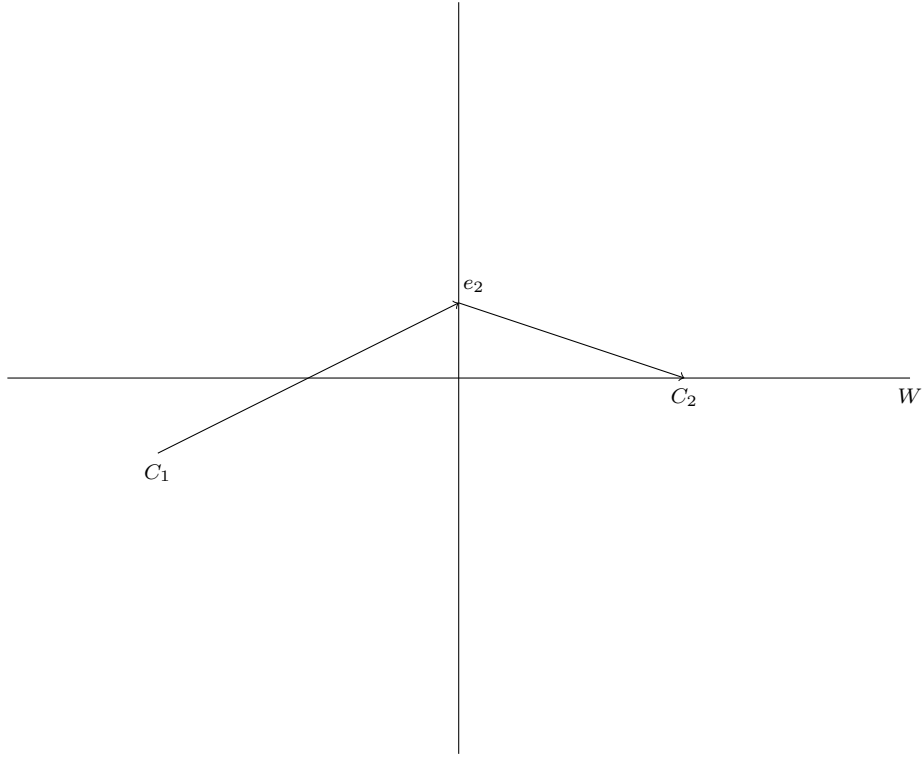
\begin{figure}[h]
 \centering
\begin{tikzpicture}[scale=1]  
\draw(-6,0)--(6,0);\node [below] at (6,0) {\tiny$W$}; 
\draw(0,-5)--(0,5);
\draw[->](-4,-1)--(0,1); \node [below] at (-4,-1) {\tiny$C_1$}; \node [above] at (0.2,1) {\tiny$e_2$};
\draw[->](0,1)--(3,0); \node [below] at (3,0) {\tiny$C_2$}; 
\end{tikzpicture}
\caption{The transversal algorithm $s(C_1,C_2)(t)$ in $\mathbb{R}^2$ with respect to the 1-codimension submanifold $W=\{(a,0):~a\in\mathbb{R}\}$.}
 \label{algorithm1}
\end{figure}

\begin{example}\label{exam:rd-spheres}
    For the pair $\left(\mathbb{R}^{d+1},\bigsqcup_{i=1}^nS^d_i(0)\right)$ where $S^d_i(0)$ is the $d$-dimensional sphere with center $0$ and radius $i$, $i=1,\ldots,n$, then we have that the transversal topological complexity $\text{TTC}(\mathbb{R}^{d+1},\bigsqcup_{i=1}^nS^d_i(0))=1$. To check this, we have that the map $s:\mathbb{R}^{d+1}\times \mathbb{R}^{d+1}\to \left(\mathbb{R}^{d+1}\right)^{[0,1]}$ given by $$s(C_1,C_2)(t)=\begin{cases}
        (1-2t)C_1,&\hbox{for $0\leq t\leq 1/2$;}\\
        (2t-1)C_2,&\hbox{for $1/2\leq t\leq 1$;}\\
    \end{cases}$$ defines a global continuous transversal algorithm in $\mathbb{R}^{d+1}$ with respect to $\bigsqcup_{i=1}^nS^d_i(0)$ (see Figure~\ref{f1}). 
\end{example}

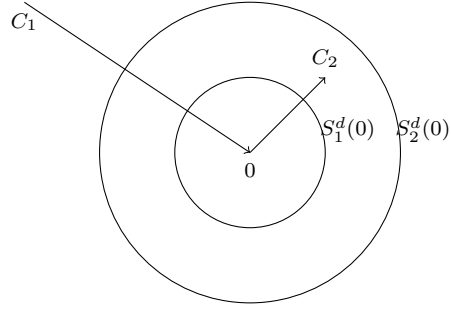
\begin{figure}[h!]
    \centering
\begin{tikzpicture}[scale=1]  
\draw (1,0) arc (0:360:1 and 1); \node [above] at (1.3,0) {\tiny$S^d_1(0)$};
\draw (2,0) arc (0:360:2 and 2); \node [above] at (2.3,0) {\tiny$S^d_2(0)$};
\draw[->](-3,2)--(0,0); \node [below] at (-3,2) {\tiny$C_1$}; \node [below] at (0,0) {\tiny$0$};
\draw[->](0,0)--(1,1); \node [above] at (1,1) {\tiny$C_2$};
\end{tikzpicture}
\caption{Global transversal algorithm in $\mathbb{R}^{d+1}$ with respect to $S^d_1(0)\sqcup S^d_2(0)$. }
\label{f1}
\end{figure}

Let $X$ be a smooth manifold, $Z\subset X$ be a 1-codimension submanifold and $x_0\in X\setminus Z$. We have the following statement.
 
 \begin{lemma}\label{carac}
 There exists a continuous transversal algorithm $s:X\times X\to X^{[0,1]}$ with respect to $Z$ if and only if there exists a continuous nulhomotopy $H:X\times [0,1]\to X$ with $H_0=1_X$, $H_1=\overline{x_0}$ and $H(x,-)\pitchfork_s Z$ for any $x\in X$.
\end{lemma}
\begin{proof}
 Suppose that $s:X\times X\to X^{[0,1]}$ is a continuous transversal algorithm with respect to $Z$. The map $H:X\times [0,1]\to X$ given by $H(x,t)=s(x,x_0)(t)$ defines a continuous nulhomotopy with $H_0=1_X$, $H_1=\overline{x_0}$ and $H(x,-)\pitchfork_s Z$ for any $x\in X$.

Now, suppose that $H:X\times [0,1]\to X$ is a continuous nulhomotopy with $H_0=1_X$, $H_1=\overline{x_0}$ and $H(x,-)\pitchfork_s Z$ for any $x\in X$. The map $s:X\times X\to X^{[0,1]}$ given by \[s(x,y)(t)=\begin{cases}
        H(x,2t),& \hbox{for $0\leq t\leq 1/2$;}\\
        H(y,2-2t),& \hbox{for $1/2\leq t\leq 1$;}\\
    \end{cases}\] defines a continuous transversal algorithm with respect to $Z$ (here we use that $x_0\notin Z$).
\end{proof}

Lemma~\ref{carac} implies the notion of transversal LS category.

\begin{definition}\label{def:tcat}
    Let $X$ be a smooth manifold and $Z\subset X$ be a 1-codimension submanifold. The \textit{transversal LS category} of $X$ with respect to $Z$, denoted by $\text{Tcat}(X,Z)$, is the smallest positive integer $k$ for which the space $X$ is covered by $k$ open subsets $X=U_1\cup\cdots\cup U_k$ such that for each $i=1,2,\ldots,k$ there exists a continuous map $H:U_i\times [0,1]\to X$ with $H_0=incl_{U_i}$, $H_1=constant$ and $H(x,-)\pitchfork_s Z$ for any $x\in U_i$. If no such $k$ exists, we set Tcat$(X,Z)=\infty$. 
\end{definition}

Note that $\text{Tcat}(X,Z)\geq \text{cat}(X)$, where $\text{cat}(X)$ is the LS category of $X$. From~\cite{cornea2003lusternik}, the \textit{LS category} of $X$ is the least integer $k$ such that $X$ can be covered by $k$ open sets, all of which are contractible within $X$. 

\begin{example}\label{tcat}
Consider the pair $(\mathbb{R}^2,S^1_1\sqcup S^1_2)$, where $S^1_1$ is the $1$-sphere with center $O_1=(-2,0)$ and $S^1_2$ is the $1$-sphere with center $O_2=(2,0)$ (see Figure~\ref{figure:two-spheres}). We have that $\text{Tcat}(\mathbb{R}^2,S^1_1\sqcup S^1_2)=2$.
  We will prove $\text{Tcat}(\mathbb{R}^2,S^1_1\sqcup S^1_2)\leq2$, the lower bound is a technical exercise which we leave to the reader. Consider the open sets:
   \begin{align*}
       U_1&=\pi^{-1}\left(-\infty,1/2\right),\\
       U_2&=\pi^{-1}\left(-1/2,+\infty\right),
   \end{align*} where $\pi:\mathbb{R}^2\to \mathbb{R}$ is the natural projection $\pi(x,y)=(x,0)$. Note that $U_1\cup U_2=\mathbb{R}^2$. Moreover, for each $i=1,2$, we can consider the map $H_i:U_i\times [0,1]\to \mathbb{R}^2$ given by \begin{align*}
       H_i(x,t)&=(1-t)x+tO_i.
   \end{align*} Note that $H_i(x,0)=x$, $H(x,1)=O_i$ and $H_i(x,-)\pitchfork_s S^1_1\sqcup S^1_2$ for any $x\in U_i$. Thus, $\text{Tcat}(\mathbb{R}^2,S^1_1\sqcup S^1_2)\leq 2$ and therefore $\text{Tcat}(\mathbb{R}^2,S^1_1\sqcup S^1_2)=2$.
\end{example}

\begin{figure}[h!]
    \centering
\begin{tikzpicture}[scale=1]  
\draw (-1,0) arc (0:360:1 and 1); \node [above] at (-2,1) {\tiny$S^1_1$};  \node [above] at (-2,-0.2) {\tiny$(-2,0)$};
\draw (3,0) arc (0:360:1 and 1); \node [above] at (2,1) {\tiny$S^1_2$}; \node [above] at (2,-0.2) {\tiny$(2,0)$};
\end{tikzpicture}
\caption{The plane $\mathbb{R}^2$ with two disjoint copies of the $1$-dimensional sphere $S^1_1\sqcup S^1_2$.}
\label{figure:two-spheres}
\end{figure}
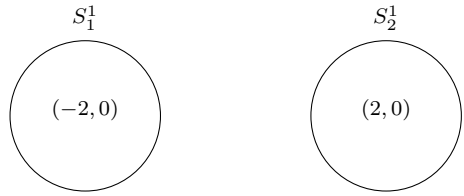

The following statement presents a lower bound for transversal complexity in terms of transversal LS category.

\begin{proposition}\label{lower-bound}
   Let $X$ be a smooth manifold and $Z\subset X$ be a 1-codimension submanifold. We have \[\text{Tcat}(X,Z)\leq \text{TTC}(X,Z).\] 
\end{proposition}
\begin{proof}
Let $x_0\in X$ and consider the map $i_0\colon X\to X\times X$, $x\mapsto i_0(x)=(x,x_0)$. For $U\subset X\times X$ and $s\colon U\to X^{[0,1]}$ satisfying $e_2^X\circ s=incl_U$ and $s(x,y)\pitchfork_s Z$ for any $(x,y)\in U$, consider $V=i_0^{-1}(U)$ and the map $H:V\times [0,1]\to X$ given by $H(x,t)=s(x,x_0)(t)$ defines a continuous map with $H_0=incl_V$, $H_1=\overline{x_0}$ and $H(x,-)\pitchfork_s Z$ for any $x\in U$. Hence, we conclude that $\text{TTC}(X,Z)\geq \text{Tcat}(X,Z)$.  
\end{proof}

\begin{example}\label{exam:2-4}
   Consider the pair $(\mathbb{R}^2,S^1_1\sqcup S^1_2)$, where $S^1_1$ is the $1$-sphere with center $(-2,0)$ and $S^1_2$ is the $1$-sphere with center $(2,0)$. We have that the transversal topological complexity $2\leq\text{TTC}(\mathbb{R}^2,S^1_1\sqcup S^1_2)\leq 4$. In fact, the first inequality follows from Example~\ref{tcat} together with Proposition~\ref{lower-bound}. To check the second inequality, we can consider the open sets: 
   \begin{align*}
       A&=\bigcap_{i=1}^{4}\pi^{-1}\left(\mathbb{R}\setminus\{a_i\}\right),\\
       B&=\bigsqcup_{i=1}^{4}\pi^{-1}\left(a_i-1/4,a_i+1/4\right),\\
   \end{align*}
    where $\pi:\mathbb{R}^2\to \mathbb{R}$ is the natural projection $\pi(x,y)=(x,0)$ and $a_1=(-3,0)$, $a_2=(-1,0)$, $a_3=(1,0)$ and $a_4=(2,0)$. 
  
   Note that $A\cup B=\mathbb{R}^2$ and thus $\left(A\times A\right)\cup \left(A\times B\right)\cup\left(B\times A\right)\cup\left(B\times B\right)=\mathbb{R}^2\times \mathbb{R}^2$. Moreover, \begin{align*}
       A\times B&=\bigsqcup_{j=1}^{4}A\times \pi^{-1}\left(a_j-1/4,a_j+1/4\right),\\
       B\times A&=\bigsqcup_{i=1}^{4}\pi^{-1}\left(a_i-1/4,a_i+1/4\right)\times A,\\
       B\times B&=\bigsqcup_{i,j=1}^{4}\pi^{-1}\left(a_i-1/4,a_i+1/4\right)\times \pi^{-1}\left(a_j-1/4,a_j+1/4\right),\\
   \end{align*}
   
   The following maps \begin{align*}
       s\colon& A\times A\to \left(\mathbb{R}^2\right)^{[0,1]},\\
       s_{A,j}\colon& A\times \pi^{-1}\left(a_j-1/4,a_j+1/4\right)\to \left(\mathbb{R}^2\right)^{[0,1]},\\
       s_{i,A}\colon& \pi^{-1}\left(a_i-1/4,a_i+1/4\right)\times A\to \left(\mathbb{R}^2\right)^{[0,1]},\\
       s_{i,j}\colon& \pi^{-1}\left(a_i-1/4,a_i+1/4\right)\times \pi^{-1}\left(a_j-1/4,a_j+1/4\right)\to \left(\mathbb{R}^2\right)^{[0,1]},\\
   \end{align*} given by \begin{align*}
      s(C_1,C_2)(t)&=\begin{cases}
     (1-3t)C_1+3t\pi(C_1),&\hbox{ for $0\leq t\leq 1/3$;}\\  
     (2-3t)\pi(C_1)+(3t-1)\pi(C_2),&\hbox{ for $1/3\leq t\leq 2/3$;}\\  
     (3-3t)\pi(C_2)+(3t-2)C_2,&\hbox{ for $2/3\leq t\leq 1$;}\\  
   \end{cases} \\
   s_{A,j}(C_1,C_2)(t)&=\begin{cases}
     (1-3t)C_1+3t\pi(C_1),&\hbox{ for $0\leq t\leq 1/3$;}\\  
     (2-3t)\pi(C_1)+(3t-1)d_j,&\hbox{ for $1/3\leq t\leq 2/3$;}\\  
     (3-3t)d_j+(3t-2)C_2,&\hbox{ for $2/3\leq t\leq 1$;}\\  
   \end{cases} \\
   s_{i,A}(C_1,C_2)(t)&=\begin{cases}
     (1-3t)C_1+3td_i,&\hbox{ for $0\leq t\leq 1/3$;}\\  
     (2-3t)d_i+(3t-1)\pi(C_2),&\hbox{ for $1/3\leq t\leq 2/3$;}\\  
     (3-3t)\pi(C_2)+(3t-2)C_2,&\hbox{ for $2/3\leq t\leq 1$;}\\  
   \end{cases} \\
   s_{i,j}(C_1,C_2)(t)&=\begin{cases}
     (1-3t)C_1+3td_i,&\hbox{ for $0\leq t\leq 1/3$;}\\  
     (2-3t)d_i+(3t-1)d_j,&\hbox{ for $1/3\leq t\leq 2/3$;}\\  
     (3-3t)d_j+(3t-2)C_2,&\hbox{ for $2/3\leq t\leq 1$;}\\  
   \end{cases}
   \end{align*} where $d_1=d_2=(-2,0)$ and $d_3=d_4=(2,0)$, define local continuous transversal algorithms with respect to $S^1_1\sqcup S^1_2$. We have thus constructed a transversal motion planner in $\mathbb{R}^2$ with respect to $S^1_1\sqcup S^1_2$ having $4$ regions of continuity $A\times A, A\times B, B\times A, B\times B$. 
\end{example}

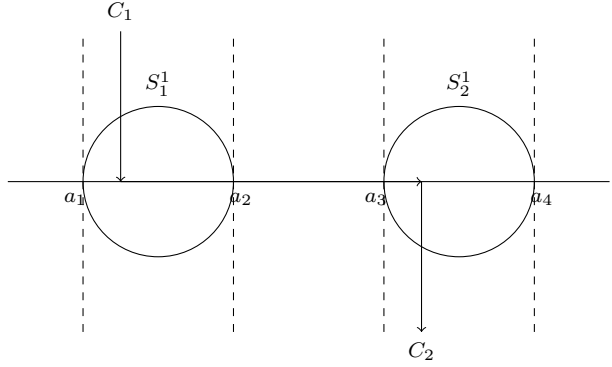
\begin{figure}[h!]
    \centering
\begin{tikzpicture}[scale=1]  
\draw (-1,0) arc (0:360:1 and 1); \node [above] at (-2,1) {\tiny$S^1_1$};
\draw (3,0) arc (0:360:1 and 1); \node [above] at (2,1) {\tiny$S^1_2$};
\draw (-4,0)--(4,0);
\draw[dashed](-3,-2)--(-3,2);\node [below] at (-3.1,0) {\tiny$a_1$};
\draw[dashed](-1,-2)--(-1,2);\node [below] at (-0.9,0) {\tiny$a_2$};
\draw[dashed](3,-2)--(3,2);\node [below] at (3.1,0) {\tiny$a_4$};
\draw[dashed](1,-2)--(1,2);\node [below] at (0.9,0) {\tiny$a_3$};
\draw[->](-2.5,2)--(-2.5,0); \node [above] at (-2.5,2) {\tiny$C_1$}; 
\draw[->](-2.5,0)--(1.5,0); 
\draw[->](1.5,0)--(1.5,-2); \node [below] at (1.5,-2) {\tiny$C_2$};
\end{tikzpicture}
\caption{The local transversal algorithm $s\colon A\times A\to \left(\mathbb{R}^2\right)^{[0,1]}$.}
\label{figure:algorithm-s}
\end{figure}

\begin{figure}[h!]
    \centering
\begin{tikzpicture}[scale=1.5]  
\draw (-1,0) arc (0:360:1 and 1); \node [above] at (-2,1) {\tiny$S^1_1$};
\draw (3,0) arc (0:360:1 and 1); \node [above] at (2,1) {\tiny$S^1_2$};
\draw (-4,0)--(4,0);
\draw[dashed](-13/4,-2)--(-13/4,2);\node [below] at (-13/4-0.2,0) {\tiny$a_1-1/4$};
\draw[dashed](-11/4,-2)--(-11/4,2);\node [below] at (-11/4+0.2,0) {\tiny$a_1+1/4$};
\draw[dashed](-5/4,-2)--(-5/4,2);\node [below] at (-5/4-0.2,0) {\tiny$a_2-1/4$};
\draw[dashed](-3/4,-2)--(-3/4,2);\node [below] at (-3/4+0.2,0) {\tiny$a_2+1/4$};
\draw[dashed](11/4,-2)--(11/4,2);\node [below] at (11/4-0.2,0) {\tiny$a_4-1/4$};
\draw[dashed](13/4,-2)--(13/4,2);\node [below] at (13/4+0.2,0) {\tiny$a_4+1/4$};
\draw[dashed](3/4,-2)--(3/4,2);\node [below] at (3/4-0.2,0) {\tiny$a_3-1/4$};
\draw[dashed](5/4,-2)--(5/4,2);\node [below] at (5/4+0.2,0) {\tiny$a_3+1/4$};
\draw[->](-2.8,2)--(-2.8,0); \node [above] at (-2.8,2) {\tiny$C_1$}; 
\draw[->](-2.8,0)--(-2,0); \node [above] at (-2,0) {\tiny$d_1=d_2$};\node [above] at (2,0) {\tiny$d_3=d_4$};
\draw[->](-2,0)--(-3,-2); \node [below] at (-3,-2) {\tiny$C_2$};
\end{tikzpicture}
\caption{The local transversal algorithm $s_{A,1}$.}
\label{figure:algorithm-s-a-j}
\end{figure}
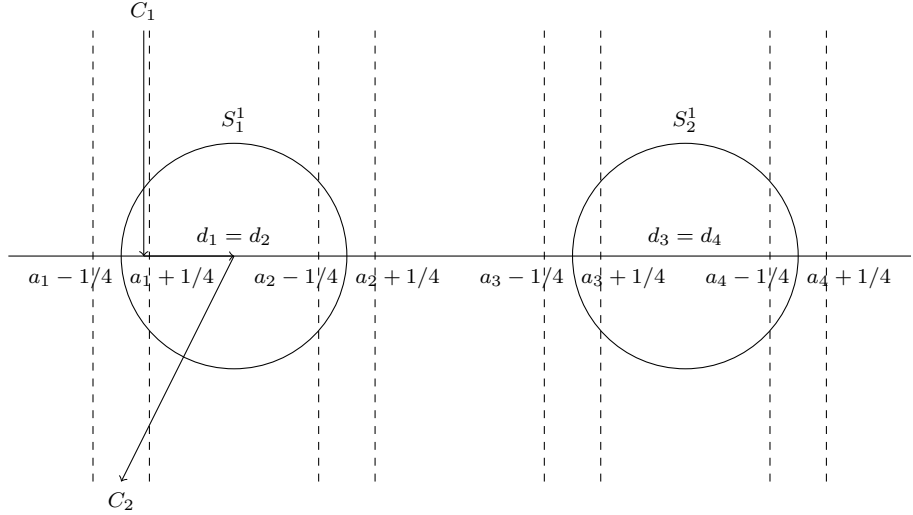

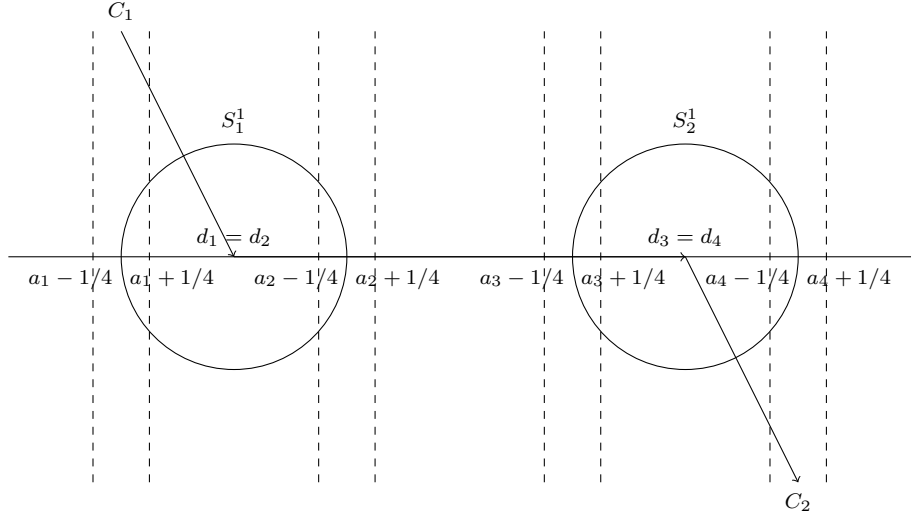
\begin{figure}[h!]
    \centering
\begin{tikzpicture}[scale=1.5]  
\draw (-1,0) arc (0:360:1 and 1); \node [above] at (-2,1) {\tiny$S^1_1$};
\draw (3,0) arc (0:360:1 and 1); \node [above] at (2,1) {\tiny$S^1_2$};
\draw (-4,0)--(4,0);
\draw[dashed](-13/4,-2)--(-13/4,2);\node [below] at (-13/4-0.2,0) {\tiny$a_1-1/4$};
\draw[dashed](-11/4,-2)--(-11/4,2);\node [below] at (-11/4+0.2,0) {\tiny$a_1+1/4$};
\draw[dashed](-5/4,-2)--(-5/4,2);\node [below] at (-5/4-0.2,0) {\tiny$a_2-1/4$};
\draw[dashed](-3/4,-2)--(-3/4,2);\node [below] at (-3/4+0.2,0) {\tiny$a_2+1/4$};
\draw[dashed](11/4,-2)--(11/4,2);\node [below] at (11/4-0.2,0) {\tiny$a_4-1/4$};
\draw[dashed](13/4,-2)--(13/4,2);\node [below] at (13/4+0.2,0) {\tiny$a_4+1/4$};
\draw[dashed](3/4,-2)--(3/4,2);\node [below] at (3/4-0.2,0) {\tiny$a_3-1/4$};
\draw[dashed](5/4,-2)--(5/4,2);\node [below] at (5/4+0.2,0) {\tiny$a_3+1/4$};
\draw[->](-3,2)--(-2,0); \node [above] at (-3,2) {\tiny$C_1$}; 
\draw[->](-2,0)--(2,0); \node [above] at (-2,0) {\tiny$d_1=d_2$};\node [above] at (2,0) {\tiny$d_3=d_4$};
\draw[->](2,0)--(3,-2); \node [below] at (3,-2) {\tiny$C_2$};
\end{tikzpicture}
\caption{The local transversal algorithm $s_{1,4}$.}
\label{figure:algorithm-s-i-j}
\end{figure}

\newpage 
\medskip The following result is technical. 

\begin{lemma}\rm{Let $h:M\to N$ be a diffeomorphism. If $W\subset N$ is a submanifold with codimension $1$, then $Z=h^{-1}(W)\subset M$ is also a submanifold with codimension $1$.}
\end{lemma}

Then, the next statement shows an explicit construction of transversal algorithms through diffeomorphisms. 

\begin{proposition}\label{prop:difeo}
Let $h\colon M\to N$ be a diffeomorphism. If $W\subset N$ is a submanifold with codimension $1$, then $\text{TTC}(N;W)=\text{TTC}(M;h^{-1}(W))$.
\end{proposition}
\begin{proof}
Let $Z=h^{-1}(W)\subset M$. For $U\subset N\times N$ and $s\colon U\to N^{[0,1]}$ satisfying $e_2^N\circ s=incl_U$ and $s(x,y)\pitchfork_s W$ whenever $(x,y)\in U$, consider $V=\left(h\times h\right)^{-1}(U)\subset M\times M$, the map $\sigma\colon V\to M^{[0,1]}$ given by $\sigma=(h^{-1})_\#\circ s\circ \left(h\times h\right)_{\mid V}$ defines a local section of $e_2^M$, where $(h^{-1})_\#:N^{[0,1]}\to M^{[0,1]}$ is the induced map of $h^{-1}$, i.e., $(h^{-1})_\#(\alpha)=h^{-1}\circ\alpha$. 

For any $(x,y)\in V$ we will show that $\sigma(x,y)\pitchfork_s Z$. First, note that $(h(x),h(y))\in U$ and thus $\Gamma = s(h(x),h(y))$ is such that $\Gamma \pitchfork_s W$. Also, note that $\sigma(x,y)=h^{-1}\circ \Gamma$ and $\sigma(x,y)(t)=h^{-1}\left(\Gamma(t)\right)$ for any $t\in [0,1]$. Suppose that $\sigma(x,y)(t_0)\in Z$ for some $t_0\in (0,1)$, then $\Gamma(t_0)\in W$ and so  $\Gamma(t)$ is smooth for $t\in (0,1)$, and $\Gamma'(t_0)\notin T_{\Gamma(t_0)}W$. Then $\sigma(x,y)(t)$ is smooth in $t_0\in (0,1)$ and $\sigma'(x,y)(t_0)=(d h^{-1})_{\Gamma(t_0)}\left(\Gamma'(t_0)\right)\notin T_{\sigma(x,y)(t_0)}Z$. Therefore, $\text{TTC}(N;W)\geq \text{TTC}(M;h^{-1}(W))$.

The inequality $\text{TTC}(M;h^{-1}(W))\geq \text{TTC}(N;W)$ follows from the first part applying to $h^{-1}\colon N\to M$ with $Z\subset M$.
\end{proof}

\begin{example}\label{exam:concrete-transversal-alg}
    Consider the diffeomorphism $h\colon \mathbb{R}^2\to \mathbb{R}^2,~h(x,y)=(x,y-x^2)$, whose inverse $f\colon \mathbb{R}^2\to \mathbb{R}^2$ is given by $f(a,b)=(a,a^2+b)$. Set $W=\{(a,b)\in\mathbb{R}^2:~b=0\}$, then $Z=f(W)=\{(x,y)\in\mathbb{R}^2:~y=x^2\}$ is the parabola. From Example~\ref{exam:transversal-rn}, we have that the map $s\colon\mathbb{R}^2\times \mathbb{R}^2\to \left(\mathbb{R}^2\right)^{[0,1]}$ given by $$s(C_1,C_2)(t)=\begin{cases}
        (1-2t)C_1+2te_2,&\hbox{for $0\leq t\leq 1/2$;}\\
       (2-2t)e_2+(2t-1)C_2,&\hbox{for $1/2\leq t\leq 1$;}\\
    \end{cases}$$ where $e_2=(0,1)$, is a transversal continuous algorithm with respect to the submanifold $W$. Then, by Proposition~\ref{prop:difeo}, the map $\sigma\colon\mathbb{R}^2\times \mathbb{R}^2\to \left(\mathbb{R}^2\right)^{[0,1]}$ given by, for $B_1=(x_1,y_1)$ and $B_2=(x_2,y_2)$: \begin{align*}
        \sigma(B_1,B_2)(t)&=f\left(s(h(B_1),h(B_2))(t)\right)\\
        &=\begin{cases}
        f\left((1-2t)h(B_1)+2te_2\right),&\hbox{for $0\leq t\leq 1/2$;}\\
       f\left((2-2t)e_2+(2t-1)h(B_2)\right),&\hbox{for $1/2\leq t\leq 1$;}\\
    \end{cases}\\
        &=\begin{cases}
        f\left((1-2t)(x_1,y_1-x_1^2)+2te_2\right),&\hbox{for $0\leq t\leq 1/2$;}\\
       f\left((2-2t)e_2+(2t-1)(x_2,y_2-x_2^2)\right),&\hbox{for $1/2\leq t\leq 1$;}\\
    \end{cases}\\
    &=\begin{cases}
        f\left((1-2t)x_1,(1-2t)(y_1-x_1^2)+2t\right),&\hbox{for $0\leq t\leq 1/2$;}\\
       f\left((2t-1)x_2,(2-2t)+(2t-1)(y_2-x_2^2)\right),&\hbox{for $1/2\leq t\leq 1$;}\\
    \end{cases}\\
     &=\begin{cases}
        \left((1-2t)x_1,(1-2t)^2x_1^2+(1-2t)(y_1-x_1^2)+2t\right),&\hbox{for $0\leq t\leq 1/2$;}\\
       \left((2t-1)x_2,(2t-1)^2x_2^2+(2-2t)+(2t-1)(y_2-x_2^2)\right),&\hbox{for $1/2\leq t\leq 1$;}\\
    \end{cases},\\   
    \end{align*} defines a transversal algorithm in $\mathbb{R}^2$ with respect to the parabola $Z$.
\end{example}

  \begin{figure}[h!]
      \centering
      \includegraphics[scale=3]{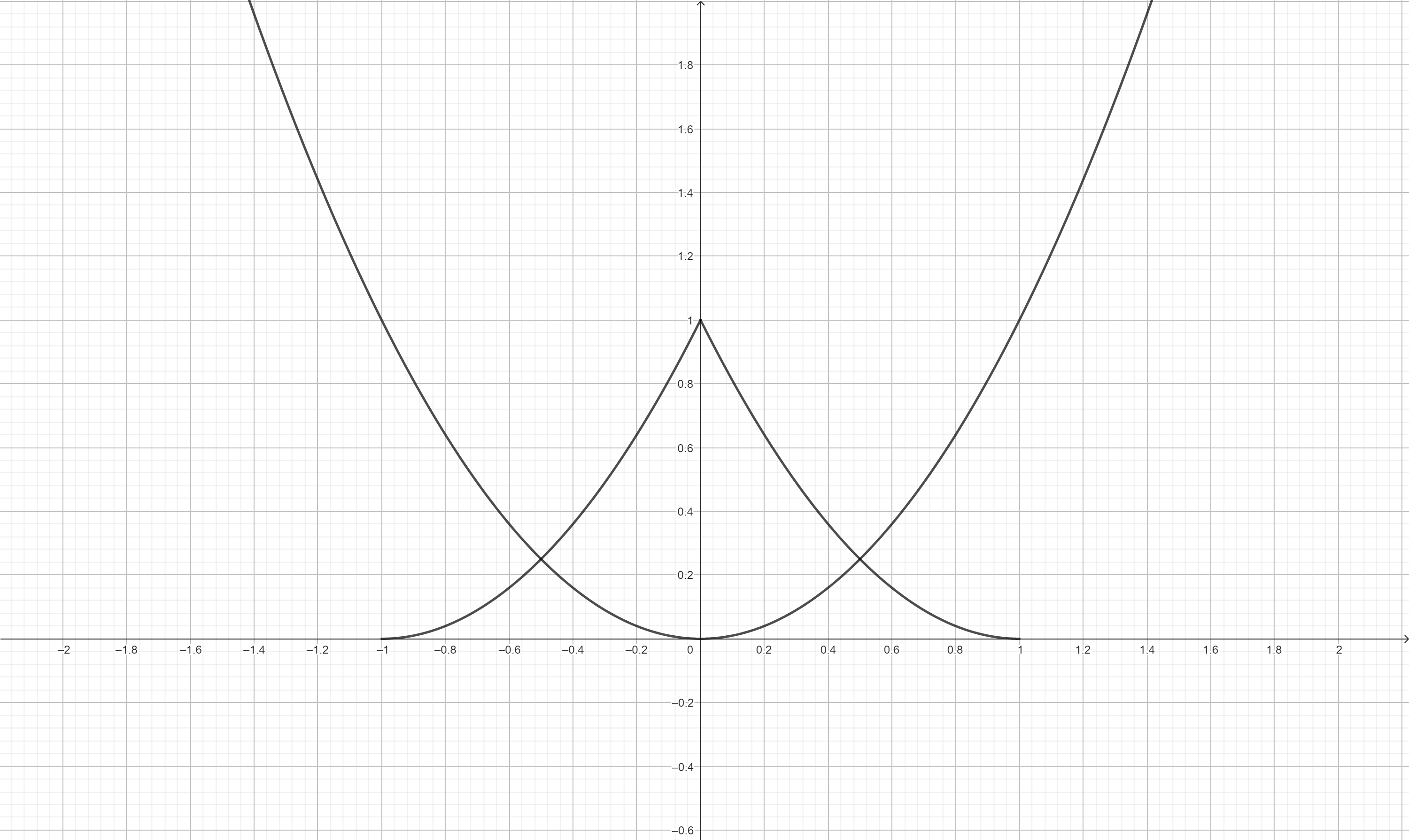}
      \caption{For $B_1=(-1,0)$ and $B_2=(1,0)$, the route given by the transversal algorithm $\sigma$ is $\sigma(B_1,B_2)(t)=\left(-(1-2t),(1-2t)^2+4t-1\right)$ for $0\leq t\leq 1/2$ and $\sigma(B_1,B_2)(t)=\left(2t-1,(2t-1)^2-4t+3\right)$ for $1/2\leq t\leq 1$.}
      \label{fig:my_label}
  \end{figure}

\bibliographystyle{plain}

\end{document}